\documentclass[11pt,a4paper]{article}
\usepackage{a4wide}
\usepackage{amsfonts,amsmath,amssymb,amsthm}
\usepackage[T1]{fontenc}
\usepackage{verbatim}

\theoremstyle{plain}
\newtheorem{theorem}{Theorem}[section]
\newtheorem{proposition}[theorem]{Proposition}
\newtheorem{corollary}[theorem]{Corollary}
\newtheorem{lemma}[theorem]{Lemma}

\theoremstyle{definition}

\newtheorem{example}[theorem]{Example}
\newtheorem{remark}[theorem]{Remark}
\newtheorem{question}[theorem]{Question}

\newcommand{\integers}{\mathbb{Z}}
\newcommand{\posintegers}{\mathbb{N}}

\newcommand{\Id}{\operatorname{Id}}
\newcommand{\Nil}{\operatorname{Nil}}
\newcommand{\id}{\operatorname{id}}
\newcommand{\krogec}{\circ}

\title{Rings in which nilpotents form a subring}
\date{October 2015}
\author{Janez \v{S}ter\\Faculty of Mechanical Engineering\\University of Ljubljana, Slovenia\\{\small e-mail: janez.ster@fs.uni-lj.si}}

\begin{document}

\maketitle

\begin{abstract}
Let $R$ be a ring with the set of nilpotents $\Nil(R)$.
We prove that the following are equivalent:
(i) $\Nil(R)$ is additively closed,
(ii) $\Nil(R)$ is multiplicatively closed and $R$ satisfies K\"othe's conjecture,
(iii) $\Nil(R)$ is closed under the operation $x\krogec y=x+y-xy$,
(iv) $\Nil(R)$ is a subring of $R$.
Some applications and examples of rings with this property are given, with
an emphasis on certain classes of exchange and clean rings.
\end{abstract}

\section{Introduction}

Rings in which nilpotents form a subring (we will call these rings \textit{NR rings} hereafter)
are closely related to Armendariz rings and their variations.
A ring $R$ is called \textit{Armendariz} if,
given any two polynomials
$f(x)=a_0+a_1x+\ldots+a_mx^m$ and
$g(x)=b_0+b_1x+\ldots+b_nx^n$ over $R$,
$f(x)g(x)=0$ implies $a_ib_j=0$ for all $i$ and $j$.
Antoine \cite{antoine} studied the structure
of the set of nilpotents in Armendariz rings and proved that in these rings nilpotents
always form a subring. He also proved in \cite{antoine} that the same is true under
a slightly weaker condition that the ring is \textit{nil-Armendariz}.
Some other results relating the Armendariz and NR conditions can be also found
in \cite{chunjeonkangleelee}, and recently \cite{chen2}.

In this paper we prove some results which concern the question
of when the set of nilpotents
in a ring is a subring in general, not in connection with any of the above mentioned
Armendariz conditions.
Roughly speaking, our main theorem shows that
the set of nilpotents $\Nil(R)$ of a ring $R$ is a subring
whenever $R$ satisfies K\"othe's conjecture (which is a weak assumption) and
$\Nil(R)$ is closed under \textit{any} of the most commonly used operations in rings,
such as addition, multiplication, the ``circle''
operation $x\krogec y=x+y-xy$, the commutator operation $(x,y)\mapsto xy-yx$, or even some more
(see Theorem \ref{glavni} and Remark \ref{glavnaopomba}).
Therefore, very little needs to be assumed for the set of nilpotents to be a subring.

Let us provide here some notations and conventions used throughout the paper.
All rings in the paper are associative and not necessarily unital, unless otherwise stated.
For a ring $R$ and for any two elements $x,y\in R$, we denote $x\krogec y=x+y-xy$.
Then $(R,\krogec)$ is a monoid. We denote by $Q(R)$ its group of units, i.e.
$$Q(R)=\{q\in R|\textnormal{ there exists $r\in R$ such that $q\krogec r=r\krogec q=0$}\}.$$
If $R$ is unital, then the group of units of $R$,
denoted $U(R)$, is precisely $U(R)=1-Q(R)$.

We denote by $J(R)$, $\Nil^*(R)$, and $\Nil(R)$ the Jacobson radical,
the upper nilradical, and the set of nilpotents of a ring $R$,
respectively. Note that $\Nil^*(R)\subseteq J(R)\subseteq Q(R)$
and $\Nil^*(R)\subseteq\Nil(R)\subseteq Q(R)$.
The ring $R$ is said to be of \textit{bounded index}
if there exists a positive integer $n$ such that $x^n=0$ for all $x\in\Nil(R)$,
and $R$ is of bounded index $n$ if $n$ is the least integer with this property.
We call $R$ a \textit{NI} ring if $\Nil(R)$ is an ideal in $R$,
and a \textit{NR} ring if $\Nil(R)$ is a subring of $R$ \cite{chunjeonkangleelee}.

Our main result will be applied to certain classes of exchange rings.
For a ring $R$, we denote by $\Id(R)$ its set of idempotents.
We say that $R$ is an \textit{exchange}
ring if for every $a\in R$ there exists $e\in\Id(R)$
such that $e=ra=s\krogec a$ for some $r,s\in R$ \cite{ara}.
If $R$ has a unit, this is equivalent to saying that for every $a\in R$
there exists $e\in\Id(R)$ such that $e\in Ra$ and $1-e\in R(1-a)$ (see \cite{ara}).

Throughout the text, $\integers$ will denote the set of integers
and $\posintegers$ the set of positive integers. $R[x]$ and
$M_n(R)$ stand for the polynomial ring and the ring of $n\times n$ matrices over $R$,
respectively.

\section{The Main Result}

In this section we prove our main theorem.
For every nilpotent $x\in\Nil(R)$, we define the \textit{index} of $x$ as
the smallest positive integer $n$ such that $x^n=0$.

\begin{lemma}
\label{redi2plus}
In any ring $R$,
if $x,y\in\Nil(R)$ are of index $2$
and $x+y\in\Nil(R)$, then $xy\in\Nil(R)$.
\end{lemma}

\begin{proof}
Let $x,y$ be elements with the given properties.
Since $x+y\in\Nil(R)$, $xy+yx=(x+y)^2\in\Nil(R)$.
We have $(xy+yx)^k=(xy)^k+(yx)^k$ for every $k$,
hence $(xy)^k+(yx)^k=0$ for some $k$, which gives $(xy)^{k+1}=0$.
\end{proof}

\begin{lemma}
\label{redi2krogec}
In any ring $R$,
if $x,y\in\Nil(R)$ are of index $2$
and $x\krogec y\in\Nil(R)$, then $xy\in\Nil(R)$.
\end{lemma}

\begin{proof}
Let $x,y$ be elements with the given properties.
First note that we can embed $R$ into a unital ring,
so that we may assume that $R$ is already unital.

Let $z=x\krogec y\in\Nil(R)$ and $u=1-z\in U(R)$.
A direct verification shows that $z^2=(xy+yx)u$.
Since $z$ and $u$ commute and $z$ is a nilpotent,
it follows that $xy+yx$ is a nilpotent.
Since $(xy+yx)^k=(xy)^k+(yx)^k$ for every $k$,
it follows that $(xy)^k+(yx)^k=0$ for some $k$,
which gives $(xy)^{k+1}=0$.
\end{proof}

We say that a ring $R$ satisfies
\textit{K\"othe's conjecture} if every nil left ideal of $R$
is contained in a nil two-sided ideal.
Many other equivalent formulations exist;
one of them, for example, is that the sum of
two nil left ideals in $R$ is a nil left ideal
(see, for example, \cite[p.~164]{lam2}).
Whether or not every ring satisfies K\"othe's
conjecture is a long-standing open question.
An important fact that will be needed in our
computations is that the rings we are dealing
with all satisfy K\"othe's conjecture.

\begin{lemma}
\label{kothe}
Let $R$ be a ring such that the set of nilpotents
$\Nil(R)$ is closed under $+$ or $\krogec$.
Then $R$ satisfies K\"othe's conjecture.
\end{lemma}

\begin{proof}
It suffices to prove that the sum of
two nil left ideals in $R$ is a nil left ideal.
If $\Nil(R)$ is closed under $+$, this is trivial,
so let us assume that $\Nil(R)$ is closed under $\krogec$.
Let $I$ and $J$ be nil left ideals of $R$,
and take $x\in I$ and $y\in J$.
We wish to prove that $x+y$ is nilpotent.
Choose $n\in\posintegers$ such that $y^n=0$, and define
$z=x+yx+y^2x+\ldots+y^{n-1}x$.
Then $z-yz=x$, and hence $x+y=y\krogec z$.
Since $I$ is a nil left ideal, $z\in I$ is nilpotent.
Hence the conclusion follows from the fact that
$\Nil(R)$ is closed under $\krogec$.
\end{proof}

Recall that the upper radical of a ring $R$,
$\Nil^*(R)$, is the sum of all nil (two-sided) ideals.
If $R$ satisfies K\"othe's conjecture then
every nil left (or right) ideal in $R$ is contained
in $\Nil^*(R)$. In particular,
if $\Nil^*(R)=0$ and $R$ satisfies K\"othe's conjecture,
then $R$ contains no nonzero nil left ideals.

\begin{lemma}
\label{glavnalema}
Let $R$ be a ring such that $\Nil^*(R)=0$ and
$R$ satisfies K\"othe's conjecture,
and suppose that $xy\in\Nil(R)$ for all $x,y\in\Nil(R)$ with $x^2=0$.
Let $r\in\posintegers$ and $x_1,x_2,\ldots,x_r\in\Nil(R)$
with $x_1^n=0$ for some $n\in\posintegers$. Then
$$x_1^{n_1}x_2x_1^{n_2}x_3\ldots x_1^{n_{r-1}}x_rx_1^{n_r}=0$$
for all $n_i\in\posintegers$ with $n_1+\ldots+n_r\ge n$.
In particular, $\Nil(R)$ is multiplicatively closed.
\end{lemma}

\begin{proof}
We will prove the statement by induction on $r$.
If $r=1$ there is nothing to prove.
Suppose that $r\ge 2$, and let $x_1,\ldots,x_r\in\Nil(R)$ with $x_1^n=0$,
and $n_1,\ldots,n_r\in\posintegers$ with $\sum n_i\ge n$.
Write $z=x_1^{n_1}x_2x_1^{n_2}\ldots x_{r-1}x_1^{n_{r-1}}$.
By the induction hypothesis we have $zx_1^{n_r}=0$.
Choose any $t\in R$. Then $(x_1^{n_r}tz)^2=0$, so that
the hypothesis of the lemma yields
$x_1^{n_r}tzx_r\in\Nil(R)$, i.e.~$tzx_rx_1^{n_r}\in\Nil(R)$.
Therefore $zx_rx_1^{n_r}$ generates a nil left ideal in $R$.
(Note that in any ring $R$, if $x\in R$ satisfies $Rx\subseteq\Nil(R)$,
then the left ideal generated by $x$, which is $\integers x+Rx$, is also nil.)
Since $R$ contains no nonzero nil left ideals, it follows that
$zx_rx_1^{n_r}=0$, which is exactly what was to be proved.

To prove the last part of the lemma,
take arbitrary $x,y\in\Nil(R)$, with $x^n=0$.
Set $r=n$, $x_1=x$, $x_2=\ldots=x_r=y$, and $n_1=\ldots=n_r=1$.
The statement of the lemma then gives $(xy)^{n-1}x=0$,
hence $(xy)^n=0$, as desired.
\end{proof}

\begin{lemma}
\label{glavnalema2}
Let $R$ be a ring such that $\Nil^*(R)=0$ and
$R$ satisfies K\"othe's conjecture,
and suppose that $xy\in\Nil(R)$ for all $x,y\in\Nil(R)$ with $x^2=y^2=0$.
Then $\Nil(R)$ is multiplicatively closed.
\end{lemma}

\begin{proof}
According to Lemma \ref{glavnalema}, we only need to prove
that $xy\in\Nil(R)$ whenever $x,y\in\Nil(R)$ with $x^2=0$.
In fact, we will prove the following: if $x^2=0$ and $y^{2^n}=0$
for some $n\in\posintegers$, then $(xy)^{2^{n-1}}x=0$.

Let us prove the statement by induction on $n$.
First, let $n=1$, i.e.~$x^2=y^2=0$. Then $(xtx)^2=0$ for every $t\in R$,
hence $xtxy\in\Nil(R)$ by the hypothesis of the lemma.
Thus $txyx\in\Nil(R)$ and hence $xyx$ generates a nil left ideal in $R$.
Since $\Nil^*(R)=0$, it follows that $xyx=0$, as desired.

Now let $n\ge 2$, and suppose that $x^2=y^{2^n}=0$.
By the inductive hypothesis we have $(xy^2)^{2^{n-2}}x=0$, hence
$(yxy)^{2^{n-2}+1}=y(xy^2)^{2^{n-2}}xy=0$,
hence $(yxy)^{2^{n-1}}=0$, and hence $(xyxy)^{2^{n-2}}x=0$ by
the induction hypothesis, i.e.~$(xy)^{2^{n-1}}x=0$. This completes the proof.
\end{proof}

Now we are in a position to give the main theorem.

\begin{theorem}
\label{glavni}
Let $R$ be a ring with the set of nilpotents $\Nil(R)$.
The following are equivalent:
\begin{enumerate}
\item
$\Nil(R)$ is additively closed.
\item
$\Nil(R)$ is multiplicatively closed and $R$ satisfies
K\"othe's conjecture.
\item
$\Nil(R)$ is closed under $\krogec$.
\item
$\Nil(R)$ is a subring of $R$.
\end{enumerate}
\end{theorem}

\begin{proof}
(iv) $\Rightarrow$ (i), (ii), (iii) is trivial.

(i) $\Rightarrow$ (iv):
Let $R$ satisfy (i). Denote the factor ring $R'=R/\Nil^*(R)$.
Then $R'$ also satisfies (i) (indeed, $\Nil(R')$ is nothing but
the set of all $x+\Nil^*(R)$ with $x\in\Nil(R)$).
By Lemmas \ref{redi2plus} and \ref{kothe},
$R'$ satisfies the hypotheses of Lemma \ref{glavnalema2},
and hence $\Nil(R')$ is multiplicatively closed.
Hence $\Nil(R)$ is multiplicatively closed as well.

(ii) $\Rightarrow$ (iv):
Suppose that $R$ satisfies (ii).
Similarly as above, we see that $R'=R/\Nil^*(R)$ also satisfies
(ii) (note that $R'$ satisfies K\"othe's conjecture since $R$ does).
Take any $x,y\in\Nil(R')$, it suffices to see that $x+y\in\Nil(R')$.
Choose $m$ and $n$ such that $x^m=y^n=0$.
By expanding the expression $(x+y)^{mn}$,
we see that each term in this expression
either contains $y^n$ (and is therefore zero)
or contains $x$ at least $m$ times (and is therefore also zero
by Lemma \ref{glavnalema}). Thus $(x+y)^{mn}=0$. Hence $\Nil(R')$
is a subring of $R'$, which proves that $\Nil(R)$ is a subring of $R$.

(iii) $\Rightarrow$ (iv):
It suffices to prove (iii) $\Rightarrow$ (ii).
Let $R$ be a ring satisfying (iii).
Similarly as before, denote $R'=R/\Nil^*(R)$,
and observe that $R'$ also satisfies (iii).
By Lemmas \ref{redi2krogec} and \ref{kothe},
$R'$ satisfies the hypotheses of Lemma \ref{glavnalema2},
from which the conclusion is immediate.
\end{proof}

It is a natural question if the assumption `$R$ satisfies K\"othe's
conjecture' in (ii) of the above theorem is actually superfluous.
Although the assumption that $\Nil(R)$ is multiplicatively closed
seems to be quite restrictive,
we have not been able to prove that in that case, $R$ actually
satisfies K\"othe's conjecture.
Providing a counterexample to this problem might even be
more difficult since such an example would certainly settle
K\"othe's conjecture in the negative.

\begin{question}
\label{vprasanje1}
Let $R$ be a ring such that $\Nil(R)$ is multiplicatively closed.
Does $R$ satisfy K\"othe's conjecture?
\end{question}

\begin{remark}
To answer Question \ref{vprasanje1}, one may assume that
$R$ is a radical ring (i.e.~$R=J(R)$).
Indeed, if $R$ is a ring with $\Nil(R)$ multiplicatively closed,
and $I,J$ are two nil left ideals in $R$ such that $I+J$ is not nil,
then $I$ and $J$ are nil left ideals in $J(R)$,
which is a radical ring with $\Nil(J(R))$ multiplicatively closed.
\end{remark}

\begin{remark}
\label{glavnaopomba}
In Theorem \ref{glavni}, we could add some more equivalent statements.
For example, one such statement, equivalent to (i)--(iv) of Theorem \ref{glavni},
would be that $\Nil(R)$ is closed under
the operation $x\ast y=x+y+xy$. (In fact, this can be easily seen to be
equivalent to (iii) of Theorem \ref{glavni}.)
Moreover, another statement would be that $\Nil(R)$
is closed under the operation $(x,y)\mapsto xy+yx$,
or under the operation $(x,y)\mapsto[x,y]=xy-yx$,
and $R$ satisfies K\"othe's conjecture. (In order to prove this,
note that if $x,y\in\Nil(R)$ are of index $2$ and $xy+yx\in\Nil(R)$ (or $xy-yx\in\Nil(R)$),
then $xy\in\Nil(R)$. For the rest of the proof, use Lemma \ref{glavnalema2}.)
Similarly as above, we do not know if
the `K\"othe conjecture' assumption in these cases is superfluous or not.
\end{remark}

The most interesting part of Theorem \ref{glavni} for us will be
the equivalence (iii) $\Leftrightarrow$ (iv).
In the following we provide a few corollaries of that equivalence.
Following \cite{chunjeonkangleelee}, we call a ring a \textit{NR ring} if
its set of nilpotents forms a subring.

Note that in any ring $R$, $\Nil(R)$ is a subset of the group $(Q(R),\krogec)$
which is closed with respect to taking inverses.
Therefore, $\Nil(R)$ is a subgroup of $Q(R)$
if and only if it is closed under $\krogec$.

\begin{corollary}
Let $R$ be a ring. Then $R$ is a NR ring if and only if $\Nil(R)$
is a subgroup of $Q(R)$.\hfill$\qed$
\end{corollary}

In particular, when $\Nil(R)$ is the whole group $Q(R)$, we have:

\begin{corollary}
\label{torabimo}
Let $R$ be a ring such that $\Nil(R)=Q(R)$.
Then $R$ is NR.\hfill$\qed$
\end{corollary}

If $R$ is a unital ring, then saying that $\Nil(R)$ is closed under $\krogec$
is the same
as saying that the set $1+\Nil(R)=\{1+x|\ x\in\Nil(R)\}$
is closed under multiplication
(and thus forms a multiplicative subgroup of $U(R)$).
Thus, for unital rings, an equivalent form of (iii) $\Leftrightarrow$ (iv) of
Theorem \ref{glavni} might be:

\begin{corollary}
Let $R$ be a unital ring. Then $R$ is NR
if and only if $1+\Nil(R)$ is a multiplicative subgroup of $U(R)$.\hfill$\qed$
\end{corollary}

A unital ring $R$ is called a \textit{UU} ring if all units are unipotent, i.e.~$U(R)=1+\Nil(R)$.
These rings were studied in \cite{calugareanu2} and \cite{danchevlam}.

\begin{corollary}
Every UU unital ring is NR.\hfill$\qed$
\end{corollary}

In \cite{chen2}, Chen called $R$ a \textit{NDG} ring if $\Nil(R)$
is additively closed, and proved that the polynomial ring
$R[x]$ is NDG if and only if $\Nil(R)[x]=\Nil(R[x])$. (While
this result is stated only for unital rings in \cite{chen2}, the proof also works
for general rings.) By Theorem \ref{glavni}, NDG rings are just NR rings.
Thus, we obtain the following interesting criterion when the polynomial ring is NR:

\begin{corollary}
Let $R$ be a ring. Then $R[x]$ is NR if and only if
$\Nil(R)[x]=\Nil(R[x]).\hfill\qed$
\end{corollary}

In particular, if $R$ is a nil ring, then the above proposition says that
$R[x]$ is NR if and only if it is nil. It is known that the polynomial ring over a nil ring
need not be nil \cite{smoktunowicz}. Thus, we recover the observation
noted in \cite[Example 2.6]{chunjeonkangleelee}
that the polynomial ring over a NR ring need not be NR.

We close this section with a few limiting examples which show that the statements of
the lemmas applied in our main theorem cannot be much generalized.

In view of Lemmas \ref{redi2plus} and \ref{redi2krogec},
one might wonder if some kind of Theorem \ref{glavni}
would also hold elementwise. For example, one might ask,
if $x,y$ are nilpotents and $x+y$ (or $xy$ or $x\krogec y$) is a nilpotent,
is then any other among $x+y,xy,x\krogec y$ nilpotent? While Lemmas \ref{redi2plus}
and \ref{redi2krogec} say that this is actually the case if the indices
of the nilpotents are sufficiently low, it is not true in general.

\begin{example}
Let $F$ be a field and $R=M_3(F)$. Define the following elements in $R$:
$$x=\left[\begin{array}{ccc}0&1&0\\0&0&1\\0&0&0\end{array}\right],\quad
y=\left[\begin{array}{ccc}0&0&0\\1&0&0\\0&-1&0\end{array}\right],\quad
z=\left[\begin{array}{ccc}0&0&0\\0&0&0\\1&0&0\end{array}\right],\quad
w=\left[\begin{array}{ccc}1&1&0\\-1&-1&0\\0&1&0\end{array}\right].$$
Then $x,y,x+y\in\Nil(R)$ but $xy,x\krogec y\notin\Nil(R)$,
$x,z,xz\in\Nil(R)$ but $x+z,x\krogec z\notin\Nil(R)$, and
$x,w,x\krogec w\in\Nil(R)$ but $x+w,xw\notin\Nil(R)$.
\end{example}

Lemma \ref{glavnalema} raises the question if every ring $R$ in which
$\Nil(R)$ is a subring satisfies that,
whenever $x,y\in R$ are nilpotents and $x$ is of index $2$,
$xy$ is also of index at most $2$. Note that by Lemma \ref{glavnalema},
this is the case if $\Nil^*(R)=0$.
The following example shows that in general it is not true.

\begin{example}
Let $n\ge 3$ and let $T$ be the ring of strictly upper triangular $2n\times 2n$ matrices
over a field $F$.
Set $R=M_2(T)$.
Clearly, $R$ is a nil ring, i.e.~$\Nil(R)=R$. (In fact, $R$ is even nilpotent of index $2n$.)
Let $A\in T$ be the matrix with ones above the main diagonal and zeros elsewhere.
Then $A$ is of index $2n$. Letting
$x=\left[\begin{smallmatrix}0&A\\0&0\end{smallmatrix}\right]\in R$ and
$y=\left[\begin{smallmatrix}0&0\\A&0\end{smallmatrix}\right]\in R$, we see that
$x$ and $y$ are of index $2$, but
$xy=\left[\begin{smallmatrix}A^2&0\\0&0\end{smallmatrix}\right]$ is of index $n$.
\end{example}

\section{Exchange NR Rings}

In this section we study the NR condition for the class of exchange rings.
We begin with a proposition which relates the Abelian and NR conditions,
and holds for any ring, regardless of the exchange property.

Recall that a ring $R$ is \textit{Abelian} if idempotents in $R$ are central.
According to \cite{chunjeonkangleelee},
NR rings and Abelian rings in general are independent of each other.
But we have the following important fact:

\begin{proposition}
\label{abelovost}
Let $R$ be a NR ring. Then $R/\Nil^*(R)$ is Abelian.
\end{proposition}

\begin{proof}
Let $R'=R/\Nil^*(R)$ and $e\in\Id(R')$, and take any $x,y\in R'$.
Since $R'$ is a NR ring and $(x-ex)e,e(y-ye)\in\Nil(R')$,
we have $e(y-ye)(x-ex)e\in\Nil(R')$,
hence $y(xe-exe)=(y-ye)(x-ex)e\in\Nil(R')$.
Since this holds for every $y\in R'$, it follows that $xe-exe$
generates a nil left ideal in $R'$. Since $R'$ contains no nonzero
nil left ideals, it follows that $xe-exe=0$.
Similarly we prove $ex-exe=0$, and hence $ex=xe$.
\end{proof}

\begin{remark}
In the above proposition, the ring $R$ in general is not Abelian.
For example, the ring of upper triangular matrices over a field is
NI and not Abelian.
\end{remark}

As noted in \cite{chunjeonkangleelee}, NR rings in general are not NI.
For example, if $R=F\langle x,y\rangle/(x^2)$
(i.e., $R$ is the ring of formal polynomials
in noncommuting variables $x$ and $y$
over a field $F$, modulo the ideal generated by $x^2$),
then $\Nil(R)=Fx+xRx$.
So $\Nil(R)$ is a subring of $R$ with the trivial multiplication, but $yx\notin\Nil(R)$,
and hence $\Nil(R)$ is not an ideal (see \cite{chunjeonkangleelee}).

Observe that, in this example, $J(R)=0$,
hence $\Nil(R)$ is not even contained in $J(R)$.
However, if $R$ is a NR exchange ring then always $\Nil(R)\subseteq J(R)$.
This follows from \cite[Corollary 2.17]{chen2}
(note that NR rings are linearly weak Armendariz,
see \cite{chen2} for details). Alternatively,
the proof could be obtained by using Proposition \ref{abelovost}.

It would be interesting to know if NR exchange rings are actually NI.

\begin{question}
\label{vprasanje2}
Let $R$ be an exchange ring.
If $R$ is NR, is then $R$ NI?
\end{question}

\begin{remark}
To answer Question \ref{vprasanje2}, we may assume that $R$ is a radical ring
(note that radical rings are exchange \cite{ara}).
Indeed, if $R$ is a NR exchange ring, then we know that $\Nil(R)\subseteq J(R)$,
so that $J(R)$ is a NR radical ring. Assume that $\Nil(R)$ is an ideal in $J(R)$.
Then, for any $a\in R$ and $x\in\Nil(R)$, $axa\in J(R)$, hence
$(ax)^2=(axa)x\in\Nil(R)$ and therefore $ax\in\Nil(R)$, which proves that $\Nil(R)$
is also an ideal in $R$, as desired.
\end{remark}

In the following we provide a partial answer to Question \ref{vprasanje2}
for the case when $R$ is of bounded index.

\begin{lemma}
\label{pomozna}
Let $R$ be a NR ring of bounded index $n$ with $\Nil^*(R)=0$.
If $x_1,\ldots,x_n\in\Nil(R)$ are of index $2$, then $x_1\ldots x_n=0$.
\end{lemma}

\begin{proof}
By assumption we have $(x_1+\ldots+x_n)^n=0$,
i.e.
\begin{equation}
\label{enasigma}
\sum_{\sigma\in X_n}x_{\sigma(1)}\ldots x_{\sigma(n)}=0,
\end{equation}
where $X_n$ denotes the set of all functions
$\sigma:\{1,\ldots,n\}\to\{1,\ldots,n\}$.
By Lemma \ref{glavnalema} we have $x_iyx_i=0$ for each
$i$ and $y\in\Nil(R)$, thus $x_{\sigma(1)}\ldots x_{\sigma(n)}=0$
whenever $\sigma(i)=\sigma(j)$ for some $i\ne j$.
Hence (\ref{enasigma}) becomes
$$\sum_{\sigma\in S_n}x_{\sigma(1)}\ldots x_{\sigma(n)}=0,$$
where $S_n\subseteq X_n$ denotes the permutation group of $\{1,\ldots,n\}$.

Let $t\in R$.
For every $\sigma\in S_n\setminus\{\id\}$ there exists $i=i(\sigma)$ such that $\sigma(i)>\sigma(i+1)$. Denote
$$y=y(\sigma)=x_{\sigma(i)}x_{\sigma(i)+1}\ldots x_ntx_{\sigma(1)}x_{\sigma(2)}\ldots x_{\sigma(i)}.$$
Since $x_{\sigma(i)}^2=0$, $y\in\Nil(R)$, and hence
$z=z(\sigma)=x_{\sigma(i+1)+1}\ldots x_{\sigma(i)-1}y\in\Nil(R)$.
By Lemma \ref{glavnalema}, $x_{\sigma(i+1)}zx_{\sigma(i+1)}=0$, hence
$$x_1\ldots x_ntx_{\sigma(1)}\ldots x_{\sigma(n)}=x_1\ldots x_{\sigma(i+1)}zx_{\sigma(i+1)}x_{\sigma(i+2)}\ldots x_{\sigma(n)}=0.$$
Hence
$$x_1\ldots x_n t x_1\ldots x_n=\sum_{\sigma\in S_n}x_1\ldots x_ntx_{\sigma(1)}\ldots x_{\sigma(n)}=x_1\ldots x_nt\sum_{\sigma\in S_n}x_{\sigma(1)}\ldots x_{\sigma(n)}=0.$$
This shows that $x_1\ldots x_n$ generates a nil left ideal in $R$
and consequently $x_1\ldots x_n=0$.
\end{proof}

\begin{lemma}
\label{poseb}
Let $R$ be an exchange NR ring of bounded index with $\Nil^*(R)=0$.
Then $\Nil(R)=0$.
\end{lemma}

\begin{proof}
First note that we may assume that $R$ is radical.
In fact, if $R$ satisfies the hypotheses of the lemma,
then $J(R)$ is clearly a radical NR ring (with $\Nil(J(R))=\Nil(R)$) of bounded index.
To see that $\Nil^*(J(R))=0$, take a nil left ideal $I$ in $J(R)$.
Then the left ideal in $R$ generated by $I$, which is $I+RI$,
is also nil. Indeed, since $R$ is NR, it suffices to see that $ax\in\Nil(R)$ for
every $a\in R$ and $x\in I$. Now, $axa\in J(R)$, hence
$(ax)^2=(axa)x\in\Nil(R)$, which gives $ax\in\Nil(R)$, as desired.
Hence $J(R)$ indeed satisfies the hypotheses of the lemma,
and therefore we may assume that $R$ is a radical ring.

Let $x\in R$ such that $x^2=0$.
We need to prove that $x=0$.
It suffices to prove that $qx\in\Nil(R)$ for every $q\in R$.
Embed $R$ into a unital ring $S$, then $u=1-q\in U(S)$,
with $u^{-1}=1-r$, where $r\in R$ is taken such that $q\krogec r=r\krogec q=0$.
Let $n$ denote the index of $R$, and, for each $i=1,\ldots,n$, define
$x_i=u^{i-1}xu^{1-i}$.
Clearly, each $x_i$ is a nilpotent of index $2$. Moreover, $x_i\in R$
since it can be expressed in terms of $x,q,r$.
Hence by Lemma \ref{pomozna}, $x_1\ldots x_n=0$. But we have
$$x_1\ldots x_n=xuxu^{-1}u^2xu^{-2}\ldots u^{n-1}xu^{1-n}=(xu)^nu^{-n}.$$
Thus $(xu)^n=0$ and hence $xu\in\Nil(R)$, which gives $x-qx=ux\in\Nil(R)$.
Thus $qx\in\Nil(R)$, as desired.
\end{proof}

Now the following is easy:

\begin{theorem}
\label{exbounded}
Let $R$ be an exchange NR ring of bounded index.
Then $R$ is NI.
\end{theorem}

\begin{proof}
Let $R'=R/\Nil^*(R)$.
Then $R'$ is an exchange NR ring of bounded index with $\Nil^*(R')=0$,
so that $\Nil(R')=0$ by Lemma \ref{poseb}.
It follows that $\Nil(R)=\Nil^*(R)$, as desired.
\end{proof}

\begin{remark}
The assumption that $R$ is exchange is crucial
in Theorem \ref{exbounded}.
For example, the ring $R=F\langle x,y\rangle/(x^2)$
is NR of bounded index $2$ (in fact, $\Nil(R)$ has trivial multiplication),
but it is not NI.
\end{remark}

We conclude the paper by providing an interesting application of Theorem \ref{glavni}
to the class of Diesl's \textit{strongly nil clean} rings.
These rings were defined by Diesl in \cite{diesl}, as rings in which
every element can be written as the sum of
an idempotent and a nilpotent that commute.
Along with these, Diesl also defined the wider class of \textit{nil clean} rings,
which have the same definition without commutativity.
It turns out that very little can be said about nil clean rings,
according to many fundamental open questions stated in \cite{diesl}.
On the other hand, much more can be said about strongly nil clean rings.
In fact, while probably not known to the author of \cite{diesl} at that time,
these rings had been completely characterized much earlier, back in 1988,
by Hirano et al.~\cite{hiranotominagayaqub}. Hirano et al.~proved that
strongly nil clean rings (of course, they were not called this way in
\cite{hiranotominagayaqub})
are precisely those rings $R$ in which the set of nilpotents $\Nil(R)$ forms
an ideal and $R/\Nil(R)$ is a Boolean ring. Accordingly,
these rings are just those which are Boolean modulo the upper nilradical.

In \cite{hiranotominagayaqub}, the authors used Jacobson's structure theory
for primitive rings to obtain their results.
Recently, several new proofs of the above equivalence appeared,
using different techniques (see \cite{nielsen,kosanwangzhou,danchevlam}).
In the following we provide another proof of that equivalence,
which, we believe, is noteworthy because it is very
short, self-contained, and it applies also to nonunital rings.

\begin{proposition}
\label{stronglynilclean}
Given a ring $R$, the following are equivalent:
\begin{enumerate}
\item
$R$ is strongly nil clean.
\item
For every $a\in R$, $a-a^2\in\Nil(R)$.
\item
$R/\Nil^*(R)$ is Boolean.
\end{enumerate}
\end{proposition}

\begin{proof}
(iii) $\Rightarrow$ (ii) is trivial.

(i) $\Leftrightarrow$ (ii):
First, let $a\in R$ be strongly nil clean, i.e.~$a=e+q$ and $eq=qe$,
where $e\in\Id(R)$ and $q\in\Nil(R)$. Then
$a-a^2=e+q-e^2-2eq-q^2=q-2eq-q^2$,
which is clearly a nilpotent since $e$ and $q$ commute.
Conversely, if $a\in R$ is such that $(a-a^2)^n=0$, then,
embedding $R$ into a unital ring $S$ and setting $e=(1-(1-a)^n)^n$,
one easily checks that $e$ is a multiple of $a^n$ and $1-e$ is a multiple of $(1-a)^n$, so that
$e(1-e)$ and $(a-e)^n=a^n-a^ne+(ae-e)^n=a^n(1-e)+(a-1)^ne$ are both multiples of $a^n(1-a)^n=0$
and thus $e-e^2=e(1-e)=0$ and $(a-e)^n=0$.
Note that also $e\in R$ and $ae=ea$. Hence $a=e+(a-e)$ is a strongly nil clean decomposition of $a$ in $R$.

(i) and (ii) $\Rightarrow$ (iii):
This is the main part of the proof.
First, observe that every ring $R$ satisfying (ii) has $\Nil(R)=Q(R)$.
Indeed, given any $q\in Q(R)$, then,
taking $r\in R$ with $q\krogec r=r\krogec q=0$,
we easily see that $q=(q-q^2)-(q-q^2)r$,
so that $q$ must be a nilpotent since $q$ and $r$ commute and $q-q^2\in\Nil(R)$.
Now, by Corollary \ref{torabimo} $R$ is a NR ring, so that
$R'=R/\Nil^*(R)$ is NR and Abelian by Proposition \ref{abelovost}.
Accordingly, since every element in $R'$ can be expressed as the sum of
an idempotent and a nilpotent, $\Nil(R')$ forms an ideal in $R'$,
so that $\Nil(R')=0$. Hence $\Nil(R)=\Nil^*(R)$,
which immediately yields (iii).
\end{proof}

\begin{remark}
From the above proof we see that the equivalence (i) $\Leftrightarrow$ (ii)
holds also on the elementwise level, that is,
an element $a$ in a ring is strongly nil clean
(i.e., the sum of an idempotent and a nilpotent)
if and only if $a-a^2$ is a nilpotent. This shows that the condition
``$a$ is strongly clean'' in \cite[Theorem 2.1]{kosanwangzhou}
is actually superfluous, and provides an easy self-evident argument for
\cite[Theorem 2.9]{kosanwangzhou}.
\end{remark}

\section*{Acknowledgements}

The author would like to thank Nik Stopar for helpful discussions and
Pace P.~Nielsen for useful comments on the previous version of this paper.

\bibliographystyle{abbrv}
\bibliography{RIWNFS}

\end{document}